\documentclass[preprint,12pt]{elsarticle}

\usepackage{enumerate}
\usepackage{amssymb}
\usepackage{lipsum}
\usepackage[a4paper, total={6.0in, 8.4in}]{geometry}
\usepackage{booktabs}
\usepackage{amsmath,amssymb,url}
\usepackage{enumitem} 
\usepackage{graphics} 
\usepackage{array}
\usepackage{dsfont}
\usepackage[all]{xy}
\usepackage{amsthm}
\usepackage{tikz}
\usepackage{bbding}
\usepackage{graphicx}

\numberwithin{equation}{section}
\newtheorem{theorem}{Theorem}[section]
\newtheorem{lemma}[theorem]{Lemma}

\newtheorem{corollary}[theorem]{Corollary}

\newtheorem{example}[theorem]{Example}

\newcommand{\dn}{\mathord{\downarrow}\hspace{0.05em}}
\newcommand{\up}{\mathord{\uparrow}\hspace{0.05em}}
\newcommand{\uuar}{\mathord{\Downarrow}\hspace{0.05em}}

\newcommand\blfootnote[1]{%
\begingroup
\renewcommand\thefootnote{}\footnote{#1}%
\addtocounter{footnote}{-1}%
\endgroup
}

\journal{}

\begin{document}

\begin{frontmatter}



\title{Sobriety of Scott topologies under countability conditions}


\author{Zhengmao He$^\star$, Zhaochen Dong$^*$, Kaiyun Wang$^*$}
\address{$^\star$School of Sciences, Southwest Petroleum University,Chengdu 610500,Sichuan China}
\address{$^{*}$School of Mathematics and Statistics, Shaanxi Normal University, Xi'an 710119, Shaanxi, China}

\begin{abstract} In this paper, we focus on the sobriety of the Scott topology in countable case. Specifically, we show that:

(1) every countable core-compact dcpo is sober with respect to the Scott topology;

(2) the lattice of all open sets for the rational numbers space $\mathbb{Q}$ equipped with the Scott topology is not sober.
\end{abstract}
\begin{keyword}  Scott topology; sober space; core-compact; rational numbers space\\
\vspace*{0.2cm}
\end{keyword}


\end{frontmatter}
\blfootnote{$^\star$ Corresponding author}
\blfootnote{This work is supported by the National Natural Science Foundation of China (Grant nos.12471438,12601906) and the Sichuan Science and Technology Program (Grant No. 2026NSFSC0784).}
\blfootnote{E-mail address: hezhengmaomath@163.com (Z.M.He);dzc0419@gmail.com(Z.C.Dong);\\ wangkaiyun@snnu.edu.cn(K.Y.Wang).}


\section{Introduction}
\label{}

Sober spaces play an important role in both general topology and the study of non-Hausdorff spaces. In 1969, M. Hochster gave a characterization for the spectral space of a commutative ring in terms of sober spaces(\cite{U11}). Domain theory was introduced by Dana Scott to provide mathematical models for computation, especially recursive and nonterminating processes. The Scott topology is fundamental in this setting, since its open sets describe information that can be detected through finite approximations. Sobriety ensures that every irreducible closed set is represented by a unique point. For this reason, sobriety is a fundamental topological property in domain theory. A central question is when a dcpo equipped with the Scott topology is a sober space in domain theory. A classical result states that the Scott topology on every continuous dcpo and, more generally, on every quasicontinuous dcpo is sober(\cite{GG03}). To investigate which order-theoretic properties ensure that a dcpo endowed with the Scott topology is sober, researchers have constructed a series of dcpos whose Scott topologies are non-sober(\cite{U13,U14,ME182,AB75,EEE20}). More recently, the following two important results concerning the non-sobriety of the Scott topology is established:

$\bullet$ The complete Boolean algebra of all the regular open subsets in the real line equipped with the Scott topology is non-sober(\cite{EEF32}).

$\bullet$ There are two uncountable dcpos whose Scott topologies are sober, but their product equipped with the Scott topology is non-sober(\cite{EEF31}).\\
Both of the conclusions stated above concern the sobriety of the Scott topology in the uncountable case. As we know, continuity and quasi-continuity on dcpos are also among the major research topics in domain theory. Recently, countability considerations plays a significant role in domain theory. More specifically, the following two important results has been established:

$\bullet$ Every countable continuous dcpo is an algebraic dcpo(\cite{U15}).

$\bullet$ Every countable quasi-continuous dcpo is a quasialgebraic dcpo(\cite{U16}).\\
The above results indicate that countability leads to stronger properties concerning continuity in dcpos. This motivates us to investigate the sobriety of the Scott topology in the countable case.

Compactness conditions provide another approach to the sobriety of the Scott topology. Lawson, Wu and Xi proved that every core-compact well-filtered $T_{0}$ space is sober(\cite{GOY03}). This result motivates the study of sobriety under compactness assumptions when well-filteredness is not assumed. Two questions proposed by Xu and Zhao [17, Questions 3.13 and 3.14] concern whether a meet-continuous dcpo with a core-compact Scott space must be Scott sober, and whether local compactness of the Scott space of a dcpo implies sobriety.

Countability alone does not guarantee Scott sobriety because Miao, Xi, Li and Zhao [13] constructed a countable distributive complete lattice whose Scott space is non-sober. Thus, even within the countable setting, additional hypotheses are needed to obtain positive results. Our results identify how the compactness assumptions in the two questions yield sobriety in this setting and establish approximation properties underlying these conclusions.

In section 3, we investigate these questions for countable dcpos. Specifically, we prove that for a countable core-compact dcpo $P$, $\Sigma P$ is sober. Therefore, Problem 3.13 and Problem 3.14 posed by Xu and Zhao in \cite{U17} has a positive answer in the countable case. 

In \cite{AB75}, Miao, Xi, Li and Zhao present a countable distributive complete lattice $L$ whose Scott topology is non-sober. Actually, $L$ is the Dedekind-MacNeille completion of the countable dcpo $\widehat{P}$ (\cite{AB75}). Interestingly, it has been proved that for the countable dcpo $\widehat{P}$, the lower topology $\omega(\widehat{P})$ is not sober with respect to its Scott topology(\cite{U18}). Furthermore, a countable $T_{1}$ space $X$ can be reconstructed such that the open set lattice of $X$ is not Scott sober(\cite{U18}). In section 4, we prove that there exists a surjective open mapping from the rational numbers space $\mathbb{Q}$ to the space $(\widehat{P},\omega(\widehat{P}))$. It follows that the lattice of all open subsets of the rational numbers space $\mathbb{Q}$ is not Scott sober. Thus, we obtain a natural countable metrizable space whose lattice of open sets is non-sober with respect to the Scott topology.
\section{Preliminaries}
\label{}

\quad Let $P$ be a poset and $A\subseteq P$. Define $$\up A=\{x\in P\mid \exists\ a\in A, a\leq x \}$$ and $$\dn A=\{x\in P\mid \exists\ a\in A, x\leq a\}.$$
For each $x\in P$, we write $\dn\{x\}$ as $\dn x$ and $\up\{x\}$ as $\up x$. A nonempty subset $D$ of $P$ is directed provided that $\forall\ a,b\in D$, $\up a\cap \up b \cap D\neq\emptyset$. Dually, a nonempty subset $D\subseteq P$ is filtered if $\forall\ a,b\in D$, $\dn a\cap \dn b \cap D\neq\emptyset$.

\quad Let $P$ be a poset and $A\subseteq P$. An element $x$ is an upper bound for $A$ if $x\in\bigcap\limits_{a\in A}\up a$. The symbol $A^{u}$ denotes all upper bounds of $A$. We say that $x\in A$ is a minimal(maximal) element of $A$ if $\dn x\cap A=\{x\}$ ($\up x\cap A=\{x\}$). The notation ${\rm min(A)}$(${\rm max(A)}$) refers to the set of all minimal(maximal) elements of $A$.

\quad Given a poset $P$ and $A\subseteq P$, a subset $B$ is cofinal in $A$ if for every $a\in A$, there is a $b\in B\cap\up a$.

\quad A poset $P$ is called a {\em directed complete poset} ({\em dcpo}, for short) if every
directed set of $P$ has a supremum.

\quad Let $P$ be a poset and $x, y\in P$. We say that $x$ is {\em way below} $y$, in symbols $x\ll y$, iff for each directed subset $D\subseteq P$ whenever $\bigvee D$ exists, $y\leq\bigvee D$ implies $\up x\cap D\neq\emptyset$.\ If $\uuar p=\{b\in P\mid b\ll p\}$ is directed and $\bigvee \uuar p=p$\ for every \ $p\in P$, we call $P$ a {\em continuous poset}.

\quad Let $P$ be a poset and $G, H\subseteq P$. We say that $G$ {\em approximates} $H$ and write $G\ll H$, iff for each directed subset $D\subseteq P$ for which $\bigvee D$ exists, $\bigvee D\in\up H$ implies $D\cap\up G\neq\emptyset$.\ In particular, $F\ll x$ denotes $F\ll\{x\}$. A poset is a {\em quasicontinuous poset}
 if $$fin(x)=\{F\mid F\ll x\ \mbox{and}\ F\ \mbox{is finite}\}$$ is directed (in the sense that for every $F_{1},F_{2}\in fin(x)$, there is a $F_{3}\in fin(x)$ such that $F_{3}\subseteq \up F_{1}\cap \up F_{2}$) and $$\bigcap fin(x)=\up x$$\ for every \ $x\in P$.

\quad Let $P$ be a poset. A subset $U\subseteq P$ is said to be {\em Scott open} (see \cite{GG03,JG13})if

 (i) $U=\up U$, and

 (ii) for each directed subset $D$, $\bigvee D\in U$ implies
$D\cap U\neq \emptyset$, whenever $\bigvee D$ exists.\\
The collection of Scott open sets of $P$ is the Scott topology $\sigma(P)$. We write $\Sigma P$ for $(P,\sigma(P))$. The lower topology $\omega(P)$ on $P$ is the topology generated by   $$\{P\setminus\up x\mid x\in P\}.$$

\quad Let $P,Q$ be dcpos. Then the mapping $f:\Sigma P\longrightarrow \Sigma Q$ is continuous (or called Scott continuous) if and only if $f$ is monotone and for each directed set $D\subseteq P$, $f(\bigvee D)=\bigvee\limits_{d\in D}f(d)$.

\quad Let $X$ be a topological space. A nonempty subset $F$ of $X$ is {\em irreducible}, if for every pair of closed sets $A, B\subseteq X$, $F\subseteq A\cup B$ implies $F\subseteq A$ or $F\subseteq B$. A $T_{0}$ space $X$ is {\em sober} if every irreducible closed set is the closure of a singleton.

\quad A dcpo $P$ is called Scott sober if the dcpo $P$ is sober endowed with the Scott topology.

\quad Given a $T_{0}$ space $X$, the {\em specialization order} $\leq$ on $X$ is defined by
$$x\leq y \Longleftrightarrow x\in cl(\{y\}).$$
Unless otherwise stated, throughout
the paper, whenever an order-theoretic concept is mentioned in the context of a $T_{0}$ space $X$, it is to be interpreted with respect to the specialization order on $X$.

\quad Given a topological space $X$, we use $\mathcal{O}(X)$ to denote the lattice of open sets of $X$ ordered by the inclusion.

\quad A subset $A$ of a topological space $X$ is {\em saturated} if $A=\up A$. Given a topological space $X$, we shall use $K(X)$ to denote the poset of all nonempty
compact saturated subsets of $X$ with the reverse inclusion order.

\quad A topological space $X$ is called {\em locally compact} if for each $U\in \mathcal{O}(X)$ and for each $x\in U$, there exist an open set $V$ and a $K\in K(X)$ such that $x\in V\subseteq K\subseteq U$. A topological space $X$ is said to be  {\em core-compact } if  $\mathcal{O}(X)$ is continuous.

\quad A topological space $X$ is a {\em retract} of a topological space $Y$ if there are two continuous maps $f:X\longrightarrow Y$ and $g:Y\longrightarrow X$ such that $g\circ f=id_{X}$.

\section{Every countable core-compact dcpo is Scott sober}
\begin{lemma} {\rm(see \cite{GG03}) Let $P$ be a dcpo. Then $P$ is quasicontinuous if and only if $\forall\ U\in\sigma(P)$ and $x\in U$, there is a nonempty finite subset $F\subseteq U$ such that $F\ll x$.}

\end{lemma}

\begin{theorem} {\rm(see \cite{GG03}) Every quasicontinuous dcpo endowed with the Scott topology is always locally compact sober.}
\end{theorem}

\begin{theorem}
{\rm Let $P$ be a countable dcpo. If
$\Sigma P$ is core-compact, then $P$ is
quasicontinuous. Consequently, $\Sigma P$ is locally compact
sober.}
\end{theorem}

\begin{proof} Fix $x\in U\in\sigma(P)$. By core-compactness, there exists
$V\in\sigma(P)$ such that
$$
x\in V\ll_{\sigma(P)}U,
$$
where $\ll_{\sigma(P)}$ denotes the way-below relation in the
open-set lattice $(\sigma(P),\subseteq)$.

$\mathbf{Claim}$: There exists a finite set $F\subseteq U$ such that $V\subseteq\uparrow F\subseteq U$.

If $U$ is finite, take $F=U$. Otherwise, we enumerate $U=\{u_1,u_2,\ldots\}$.
Assume that no such finite set exists.
Since $U$ is an upper set, for each $n\geq1$ we can choose
$$
a_n\in V\setminus\uparrow\{u_1,\ldots,u_n\}.
$$
For each $k\geq1$, define
$$
C_k=(P\setminus U)\cup\bigcup_{n\geq k}\downarrow a_n.
$$

It is clearly a lower set. Let $D\subseteq C_k$ be directed.
If $\sup D\notin U$, then $\sup D\in C_k$.
Suppose that $\sup D\in U$. Since $U$ is Scott open,
there exists
$$
d_0=u_j\in D\cap U.
$$
The set $D'=D\cap\uparrow d_0$ is a cofinal directed subset
of $D$, and hence
$$
\sup D'=\sup D.
$$
Moreover, $D'\subseteq U$ because $U$ is an upper set.
For every $d\in D'$,  $d\in C_k\cap U$ yields
some $n\geq k$ such that $d\leq a_n$.
Thus $u_j=d_0\leq d\leq a_n$.
The choice of $a_n$ forces $n<j$, so
$$
D'\subseteq\bigcup_{k\leq n<j}\downarrow a_n.
$$
Since $D'$ is nonempty and $\bigcup\limits_{k\leq n<j}\downarrow a_n$ is Scott closed,
, we have
$$
\sup D=\sup D'
\in\bigcup_{k\leq n<j}\downarrow a_n
\subseteq C_k.
$$
This proves that $C_k$ is Scott closed.

Next, we have
$$
\bigcap_{k\geq1}C_k=P\setminus U.
$$
Indeed, $P\setminus U\subseteq\bigcap_{k\geq1}C_k$ is immediate.
For the reverse inclusion, take any $u_j\in U$.
For every $n\geq j$, the choice of $a_n$ gives
$u_j\nleq a_n$. Hence $u_j\notin C_j$.

Put $W_k=P\setminus C_k$.
Then $\{W_k\}_{k\geq1}$ is an increasing sequence of Scott-open
sets satisfying
$$
\bigcup_{k\geq1}W_k=U.
$$
Since $V\ll_{\sigma(P)}U$, there exists $k$ such that
$V\subseteq W_k$. However, $a_k\in V\cap C_k$, a contradiction.
This claim is proved.

We have therefore shown that, for every $x\in U\in\sigma(P)$,
there exist a Scott-open set $V$ and a nonempty finite set
$F\subseteq U$ such that
$$
x\in V\subseteq\uparrow F\subseteq U.
$$

By Lemma 3.1, $P$ is quasicontinuous. Hence, using Lemma 3.2, $\Sigma P$ is locally compact and sober.
\end{proof}

There exists a countable core-compact complete lattice $L$ which is continuous.
\begin{example} {\rm Let $L=\{\top,\bot\}\cup\{a_{i}\mid i\in\mathbb{N}\}\cup\{b_{j}\mid j\in\mathbb{N}\}$ ordered by $$\bot\leq a_{0}\leq a_{1}\leq\cdot\cdot\cdot\leq a_{i}\leq\cdot\cdot\cdot\leq\top\ \text{and}\ \bot\leq b_{0}\leq b_{1}\leq\cdot\cdot\cdot\leq b_{j}\leq\cdot\cdot\cdot\leq\top.$$ Then we can conclude that $L$ is a countable complete lattice. Furthermore, $\Sigma L$ is core-compact as all Scott open sets $U$ in $L$ are compact sets. However, $\Sigma L$ is not continuous because $\bigvee\uuar \top=\bigvee\{\bot\}=\bot$. }

\end{example}

\section{The lattice of open sets of the rational number space is not Scott sober.}
In the following, we assert that for rational numbers space $\mathbb{Q}$, $\Sigma\mathcal{O}(\mathbb{Q})$ is not sober.

Let $\widehat{P}$ be the countable dcpo (see \cite{AB75} Example 3.1) constructed by Miao,Xi,Li and Zhao.
\begin{lemma}{\rm (see \cite{U18}) $\widehat{X}=(\widehat{P},\omega(\widehat{P}))$ is a second countable compact $T_{0}$ space but not $SI$-compact.}
\end{lemma}

\begin{lemma}{\rm (see \cite{U19}) Let $X$ be a compact space such that $\Sigma\mathcal{O}(X)$ is sober. Then $X$ is $SI$-compact.}
\end{lemma}

By Lemma 5.1 and 5.2, the following corollary 5.3 holds.

\begin{corollary}{\rm $\Sigma\mathcal{O}(\widehat{X})$ is not sober.}

\end{corollary}

\begin{lemma}{\rm (see \cite{RE77}) Every countable dense in itself metrizable space is homeomorphic to the rational numbers space $\mathbb{Q}$.}
\end{lemma}

\begin{theorem} {\rm Let $X$ be a first-countable $T_{0}$ space with $|X|\leq\aleph_{0}$. There exists a surjective open mapping $g$ from the rational number space $\mathbb{Q}$ to $X$.}
\end{theorem}

\begin{proof} Since $X$ is first-countable, for every $x\in X$, we can fix a neighbourhood base $\mathcal{B}_{x}=\{B(x,n)\mid n\in\mathbb{N}\}$ at $x$. Set $\mathcal{B}=\bigcup\limits_{x\in \widehat{X}}\mathcal{B}_{x}$. Since $|X|\leq\aleph_{0}$, $\mathcal{B}$ is a countable topology base of $X$. Then we can enumerate $\mathcal{B}$ as $\{B_{i}\mid i\in {\rm I}\}$ for some ${\rm I}\subseteq\mathbb{N}$. Equip the index set ${\rm I}$ with the discrete topology. Then ${\rm I}^{\mathds{N}}$ is a second-countable metric space. In deed, a metric $d$ on ${\rm I}^{\mathds{N}}$ is given by
$$\forall \ a,b\in {\rm I}^{\mathds{N}}, \ d(a,b)=\left\{
             \begin{array}{ll}
              0, &\ \ a=b, \\
              2^{-k}, &\ \ k={\rm min}\{n\mid a(n)\neq b(n)\}.
             \end{array}
           \right.$$

Let $$M=\{a\in{\rm I}^{\mathds{N}}\mid \{B_{a(n)}\mid n\in\mathbb{N}\} \mbox{ is a neighbourhood base at a ponit}\ x\in X\}.$$

As $X$ is $T_{0}$, for every $a\in M$ there is a unique $x_{a}\in X$ such that $\{B_{a(n)}\mid n\in\mathbb{N}\}$ is a neighbourhood base at $x_{a}$. The mapping $p:M\longrightarrow X$ is defined by $p(a)=x_{a}$.

$\mathbf{Claim}$ 1: $p$ is surjective.

Note that for every $x\in X$, $\{B(x,n)\mid n\in\mathbb{N}\}$ is a neighborhood base at $x$. We take $b\in I^{\mathds{N}}$ such that $B_{b(n)}=B(x,n)$ for every $n\in\mathbb{N}$. Whence, $p(b)=x$. This shows that $p$ is a surjection.

$\mathbf{Claim}$ 2: $p$ is an open mapping.

Let $s=(i_{0},i_{1},i_{2},\cdot\cdot\cdot,i_{m-1})\in{\rm I}^{m}$. Set $$[s]_{M}=\{a\in M\mid a(j)=i_{j}, \mbox{for all}\ 0\leq j\leq m-1\}.$$ Then we can check that $$p([s]_{M})=\bigcap\limits_{0\leq j\leq m-1}B_{i_{j}}.$$  Indeed, $p([s]_{M})\subseteq\bigcap\limits_{0\leq j\leq m-1}B_{i_{j}}$ is clear. In addition, choose a $z\in\bigcap\limits_{0\leq j\leq m-1}B_{i_{j}}$. We fix a $k\in{\rm I}^{\mathbb{N}}$ such that $$j\in\mathds{N}, \ k(j)=\left\{
             \begin{array}{ll}
              i_{j}, &\ \ 0\leq j\leq m-1, \\
              \widehat{j}, &\ B_{\widehat{j}}=B(z,j-m),m\leq j.
             \end{array}
           \right.$$ Then $k\in[s]_{M}$ and $p(k)=z$. Note that the collection of all sets with the forms $[s]_{M}$ is a topology base of $M$. So $p$ is an open mapping.

$\mathbf{Claim}$ 3: $p$ is continuous.

Let $U$ be an open set in $X$ and $a\in p^{-1}(U)$. Suppose $p(a)=x\in U$. Since $\{B_{a(n)}\mid n\in\mathbb{N}\}$ is a neighbourhood base at $x$, there is a $q$ such that $B_{a_{q}}\subseteq U$. Consider $C=[(a(0),a(1),\cdot\cdot\cdot, a_{p})]_{M}$. Then we have $P(C)=\bigcap\limits_{0\leq j\leq q}B_{a_{j}}\subseteq B_{a_{q}}\subseteq U$. Consequently, $a\in C\subseteq p^{-1}(U)$. This means that $p^{-1}(U)$ is open in $M$. Hence, $p$ is continuous.

Note that $M$ is a second-countable metric space. We take $\mathcal{C}_{0}=\{B_{n}\mid n\in\mathbb{N}\}$ as a topology base for $M$. Set $$J=\{(B,x)\in\mathcal{C}\times X\mid x\in p(B)\},$$ where $\mathcal{C}=\mathcal{C}_{0}\cup\{M\}$.
Then $\forall\ (B,x)\in J$, $B\cap p^{-1}(x)\neq\emptyset$. Take a $e_{B,x}\in B\cap p^{-1}(x)$. Set $$E=\{e_{B,x}\mid(B,x)\in J\}.$$

$\mathbf{Claim}$ 4: $\forall B\in\mathcal{C}$, $p(E\cap B)=p(B)$.

Suppose $x\in p(B)$. Then we have $e_{B,x}\in E\cap B\cap p^{-1}(x)$. Consequently, $x=p(e_{B,x})\in p(E\cap B)$. Claim 4 is proved.

Consider the mapping $p_{E}:E\longrightarrow X$ defined by $p_{E}(e)=p(e)$.

$\mathbf{Claim}$ 5: $p_{E}$ is a surjection.

Take a $x\in X$, $x\in p(M)$ because $p$ is a surjection. As $M\in\mathcal{C}$, we have $e_{M,x}\in E\cap M\cap p^{-1}(x)$ and hence $p(e_{M,x})=p_{E}(e_{M,x})=x$.

$\mathbf{Claim}$ 6: $p_{E}$ is continuous.

Since $p:M\longrightarrow X$ is continuous, $p_{E}$ is also continuous with respect to the subspace topology $E$.

$\mathbf{Claim}$ 7: $p_{E}$ is an open mapping.

Let $V$ be an open set in $M$ and $W=E\cap V$. Then  $V=\bigcup\{B\in\mathcal{C}\mid B\subseteq V\}$. By claim 4, we have
 $$P_{E}(W)=p(E\cap V)=\bigcup\limits_{B\in\mathcal{C},B\subseteq V}p(E\cap B)=\bigcup\limits_{B\in\mathcal{C}, B\subseteq V}p(B).$$ Thus, $P_{E}(W)$ is open in $X$.

Now, we obtain a countable nonempty second-countable metric space $E$. Let $$Z=E\times \mathbb{Q}.$$ Then $Z$ is a countable nonempty second-countable metric space. Furthermore, $Z$ is dense in itself because $\mathbb{Q}$ is that space. By Lemma 5.4, $Z$ is homeomorphic to $\mathbb{Q}$ and the homeomorphism from $\mathbb{Q}$ to $Z$ is denoted by $h$. Now, we can conclude that $g:\mathbb{Q}\longrightarrow X$ is the mapping we desired, where $$g=p_{E}\circ \pi_{E}\circ h$$ and $\pi_{E}: E\times \mathbb{Q}\longrightarrow E$ is the projection.
\end{proof}

\begin{lemma}{\rm (see \cite{GG03}) Every retract of a sober space is sober.}
\end{lemma}

\begin{theorem} {\rm $\Sigma\mathcal{O}(\mathbb{Q})$ is not sober.}
\end{theorem}
\begin{proof} By Lemma 5.1 and Theorem 5.5, there exists a surjective open continuous mapping $g$ from $\mathbb{Q}$ to $\widehat{X}$. Define two mappings $$s:\mathcal{O}(\widehat{X})\longrightarrow\mathcal{O}(\mathbb{Q})$$ and $$t:\mathcal{O}(\mathbb{Q})\longrightarrow\mathcal{O}(\widehat{X})$$ by
$$s(U)=g^{-1}(U)\ \mbox{and}\ t(V)=g(V).$$

Then $s$ and $t$ are well-defined and further $s\circ t=id$. In addition, we can check that $s$ and $t$ are continuous with respect to the Scott topology. It follows that $\Sigma\mathcal{O}(\mathbb{Q})$ is a retract of $\Sigma\mathcal{O}(\widehat{X})$. By Corollary 5.3 and Lemma 5.6, $\Sigma\mathcal{O}(\mathbb{Q})$ is not sober.
\end{proof}

\end{document}